\documentclass[12pt, oneside, a4paper]{article}

\usepackage{amsmath}
\usepackage{amsfonts}
\usepackage{amssymb}
\usepackage{amsthm,mathrsfs}
\usepackage{enumerate}
\usepackage{graphicx}
\newtheorem{theorem}{Theorem}[section]

\newtheorem{lemma}[theorem]{Lemma}
\newtheorem{proposition}[theorem]{Proposition}
\newtheorem{corollary}[theorem]{Corollary}

\theoremstyle{definition}

\newtheorem{example}[theorem]{Example}
\newtheorem{definition}[theorem]{Definition}

\title{\textbf{The problem of deciding the positivity of Kronecker coefficients and Saxl conjecture}}
\author{Mahdi Ebrahimi\footnote{ m.ebrahimi.math@ipm.ir\\ School of Mathematics, Institute for Research in Fundamental Sciences (IPM) P.O. Box: 19395--5746, Tehran, Iran}
\\ Declarations of interest: none
\\
}
\date{}

\begin{document}

\maketitle


\begin{abstract}
Given a positive integer $k$, let $n:=\binom{k+1}{2}$.
 In 2012, during a talk at UCLA, Jan Saxl conjectured that all irreducible representations of the symmetric group
  $\mathfrak{S}_n$ occur in the decomposition of the tensor square of the irreducible representation corresponding to the staircase partition.
 In this paper, we investigate two useful methods to obtain some irreducible representations that occur in this decomposition. Our main tools are the semi-group property for Kronecker coefficients and generalized blocks of symmetric groups.

  \end{abstract}

\section{Introduction}
A longstanding open problem in algebraic combinatorics, known as the \textit{Kronecker problem}, is to find a combinatorial formula for Kronecker coefficients \cite{St}.
Given a partition $\lambda$ of a positive integer $n$, let $[\lambda]$ be the associated irreducible complex character of the symmetric group $\mathfrak{S}_n$.
For partitions $\alpha, \beta,\nu\vdash n$, the \textit{Kronecker coefficient} $g(\alpha, \beta, \nu)$ is the coefficient that occurs in the (Kronecker) product $[\alpha][\beta]=\sum_{\nu}g(\alpha,\beta,\nu)[\nu]$.
They may be computed via the scalar product, 
 $$\langle [\alpha][\beta],[\nu]\rangle=\frac{1}{n!}\sum_{\pi\in \mathfrak{S}_n}[\alpha](\pi)[\beta](\pi)[\nu](\pi),$$
 from which it also shows that the Kronecker coefficients are symmetric in $\alpha,\beta,\nu$. 

 These coefficients have been described as 'perhaps the most challenging' deep and secretive objects in algebraic combinatorics \cite{PP}.
 Since general satisfactory explicit formulas or combinatorial descriptions for Kronecker coefficients seem very hard to find, it is desirable to look for some kind of global behavior of these coefficients.
 For example, see \cite{Bs,Bu,Dv,Pa}
 
The problem of deciding the positivity of Kronecker coefficients born in quantum information theory \cite{CD,CH,CM}. Related to Kronecker positivity problem, there is a fundamental conjecture posed by Heide, Saxl, Tiep and Zalesskii \cite{HS}. It states that for every positive integer $n\neq 2,4,9$ there is a complex irreducible character of the symmetric group $\mathfrak{S}_n$ whose square contains all irreducible characters of $\mathfrak{S}_n$ as constituents. When $n$ is a triangular number, a candidate for this irreducible character of $\mathfrak{S}_n$ was suggested by Saxl in 2012:
For $k\geq 1$, the staircase partition $\rho_k$ of $n:=\binom{k+1}{2}$ is $\rho_k=(k, k-1,\dots,1)$. Saxl's Conjecture says that the Kronecker square $[\rho_k]^2$ contains all irreducible characters of $\mathfrak{S}_n$ as constituents.
This conjecture has inspired a lot of recent research. For instance, we refer the reader to \cite{BO,Ik,Pa}.
 We start with the following result which is a partially answer to the problem of deciding the positivity of Kronecker coefficients.

\begin{proposition}\label{block}
Given partitions $\xi ,\beta,\alpha \vdash n$, let $g(\xi,\beta,\alpha)\neq 0$. If $\xi$ is a $t$-core partition, then either 
\begin{itemize}
		\item [a)] $\alpha$ is a $t$-core partition, or 
		\item [b)] there exists a partition $\nu\neq\alpha$ of $n$ such that $\alpha$ and $\nu$ have the same $t$-core and $g(\xi,\beta,\nu)\neq 0$.
	\end{itemize}
\end{proposition}
Blasiak \cite{Bl} gave a positive combinatorial formula for the Kronecker coefficient $g(\lambda, \mu(d),\nu)$ for any partitions $\lambda, \nu$ of $n$ and hook partition $\mu(d):=(n-d,1^d)$. The following result guaranties the existence of an irreducible constituent indexed by a hook partition $\neq (n),(1^n)$ in the Kronecker product $[\alpha]^2$ of a partition $\alpha\neq (n),(1^n)$ of $ n$.
\begin{corollary}\label{hook}
    Let $\alpha\vdash n$. If $\alpha\neq (n),(1^n)$, then there is a hook partition $\beta\neq (n),(1^n)$ of $n$ such that $g(\alpha,\alpha,\beta)\neq 0$.
\end{corollary}

Let $n$ be a positive integer. We denote by $\sharp_n$ the rectangular partition $(n^n)\vdash n^2$. 
  Suppose that $\alpha=(\alpha_1,\alpha_2,\dots,\alpha_h)\vdash n$. We denote by $\mathrm{Kron}(\alpha)$, the set of all partitions $\beta\vdash n$ satisfying $g(\alpha,\alpha,\beta)\neq 0$. In addition, for every pair $(a,b)$ of non-negative integers, we define $\alpha^{(a,b)}:=(\alpha_1+a, \alpha_2, \dots, \alpha_h, 1^b)\vdash n+a+b$. To formulate the next result, we require the following definitions.

\begin{definition}
	Let $k$ and $m$ be positive integers with $k\geq m$. We define a \textit{$(k,m)$-good pair} to be a pair of non-negative integers $(a,b)$ such that $a+b=\binom{k+1}{2}-\binom{m+1}{2}$, and for some subset $I\subseteq\{m+1,m+2,\dots,k\}$, we have that $a=\sum_{i\in I}i$. We denote by $G(k,m)$, the set of  all $(k,m)$-good pairs.
	\end{definition}

We now recall \textit{semi-group property} for Kronecker coefficients \cite{Ma} which allows us to break partitions down into smaller ones.
\begin{proposition}\label{sum}(semi-group property \cite{Ma}.)
	Let $\alpha,\beta,\gamma\vdash n_1$, and $\lambda,\mu,\nu\vdash n_2$. If both $g(\alpha,\beta,\gamma)>0$ and $g(\lambda,\mu,\nu)>0$ then $g(\lambda+\alpha,\mu+\beta,\nu+\gamma)\geq \mathrm{max}\{g(\lambda,\mu,\nu),g(\alpha,\beta,\gamma)\}$.
\end{proposition}
\noindent This property leads us to the next definition.

\begin{definition}
	Let $k$ be a positive integer, and let $m:=\lfloor\log_2 k\rfloor$. We define a \textit{telescopic partition}  of $ \binom{k+1}{2}$ to be a partition $\alpha\vdash \binom{k+1}{2}$ in which $\alpha$ can be  written as $$\alpha=(\beta_s+\sum_{i=s}^{t}(\sharp_{2^i}+\alpha_i))^{(a,b)},$$ where $s$ and $t$ are integers with $0\leq s\leq t\leq m-1$, the pair $(a,b)$ is $(k,2^{t+1})$-good, $\beta_s\in \mathrm{kron}(\rho_{2^s})$ and for every integer $s\leq i\leq t$, $\alpha_i\in \mathrm{kron}(\rho_{2^i})$. We use the notation $\mathrm{Par}_{T}(k)$ to indicate the set of all telescopic partitions of $\binom{k+1}{2}$.
\end{definition}

\begin{theorem}\label{saxl}
	
	Let $k$ be a positive integer. 
	\begin{itemize}
		\item [a)] For every $\alpha\in \mathrm{Par}_{T}(k)$, the Kronecker coefficient  $g(\rho_k,\rho_k,\alpha)$ is non-zero.
		\item [b)] Assume that $t$ is a positive integer such that $t$ is even or $t\geq 2k+1$. If $\alpha \in \mathrm{Kron}(\rho_k)$, then either
		\begin{itemize}
			\item [i)] $\alpha$ is a $t$-core partition, or 
			\item [ii)] there exists a partition $\nu \in\mathrm{Kron}(\rho_k)\backslash\{\alpha\}$ such that $\alpha$ and 
			$\nu$ have the same $t$-core.
		\end{itemize}
	\end{itemize}
	
\end{theorem}

\begin{corollary}\label{telescope}
	Let $k$ be a positive integer and let $m:=\lfloor\log_2 k\rfloor$. Then for every $(a,b)\in G(k,2^m)$, the Kronecker coefficient $g(\rho_k,\rho_k,(\rho_1+\sum_{i=0}^{m-1}(\sharp_{2^i}+\rho_{2^i}))^{(a,b)})$ is non-zero.
\end{corollary}


\section{preliminaries}

 We let $\mathfrak{S}_n$ denote the symmetric group on $n$ letters. We will briefly review some definitions on the representation theory of symmetric groups \cite{GA}.
 A \textit{partition} $\lambda$ of a positive integer $n$, denoted $\lambda\vdash n$, is defined to be a weakly decreasing sequence $\lambda=(\lambda_1, \dots, \lambda_l)$ of positive integers such that the sum  $|\lambda|=\sum_{i=1}^{l}\lambda_i$ is equal to $n$. The \textit{length}  $l(\lambda)$ of a partition $\lambda\vdash n$ is the number of parts of $\lambda$.
For convenience, the notation $(\lambda_1^{a_1},\dots,\lambda_r^{a_r})\vdash n$ is used to denote the partition where $\lambda_i$'s are the distinct parts that occur with multiplicities $a_i$'s.
We identify a partition $\lambda$ with its associated \textit{Young diagram}, that is, the set of cells $\{(i,j)\in \mathbb{N}^2_{>0}\,|\,1\leq i\leq l(\lambda),\, j\leq \lambda_i\}$.
The \textit{conjugate} or \textit{transpose} $\lambda^\prime$ of $\lambda$ is defined to be equal to the partition obtained from $\lambda$ by reflecting its Young diagram through the $45^\circ$ diagonal. A \textit{Hook partition} of $n$ is a partition of the form $\lambda=(n-m,1^m)$, where $m\leq n$ is a non-negative integer. Given partitions $\lambda=(\lambda_1, \lambda_2,\dots,\lambda_s)\vdash m$ and $\mu=( \mu_1, \mu_2,\dots, \mu_t) \vdash n$, the partition $\lambda+\mu \vdash m+n$ is defined as follows: Without loss of generality, we can assume that $s\geq t$. The $i$-th part of $\lambda+\mu$ is given by 
$$(\lambda+\mu)_i:=\begin{cases} \lambda_i+\mu_i  & \text{if }1\leq i\leq t\\  \lambda_i    & \text{if } t+1\leq i\leq s
\end{cases}.$$

It is well known that both the conjugacy classes of $\mathfrak{S}_n$ and the irreducible characters of $\mathfrak{S}_n$ are indexed by partitions of $n$. Let $C$ be a conjugacy class of $\mathfrak{S}_n$. As all elements in a conjugacy class $C$ indexed by a partition $\alpha\vdash n$ have the same cycle structure, for every element $\pi\in C$, we say that $\pi$ is an element indexed by $\alpha$.
 The character corresponding to a partition $\lambda$ of $n$ is denoted by $[\lambda]$. As $[\lambda]$ is a class function of $\mathfrak{S}_n$, we use $[\lambda](\alpha)$ to indicate the value of $[\lambda]$ on the conjugacy class indexed by $\alpha \vdash n$. The relation between $[\lambda]$ and $[\lambda^\prime]$ is reflected in the following lemma.

\begin{lemma}\cite[2.1.8]{GA}\label{sgn}
	Let $\lambda\vdash n$. Then $[\lambda^\prime]=[(1^n)][\lambda]$.
\end{lemma}

We now present a simple observation that we will frequently use in section 3.

\begin{lemma}\label{prim}
	Let $\lambda,\mu,\nu\vdash n$. 
	\begin{itemize}
		\item [a)]  $g(\lambda^\prime,\mu^\prime, \nu)=g(\lambda,\mu, \nu)$.
		\item[b)] If $\lambda$ is self-conjugate, then $g(\lambda,\mu^\prime,\nu)=g(\lambda,\mu,\nu)$.
	\end{itemize}
\end{lemma}

\begin{proof}
	It immediately follows from Lemma \ref{sgn}.
\end{proof}
Let $m$ and $n$ be positive integers.
 Given partitions $\mu \vdash m$ and $\lambda\vdash m+n$, the partition $\mu$ is a \textit{subpartition} of $\lambda$, written $\mu \subseteq \lambda$, if $l(\mu)\leq l(\lambda)$ and $\mu_i\leq \lambda_i$, for all $1\leq i\leq l(\mu)$.
 The \textit{skew diagram} $\lambda/\mu$ is defined as the set of cells in $\lambda$ but not in $\mu$.
 We use the notation $\mathrm{ht}(\lambda/\mu)$ for the number of rows of $\lambda/\mu$ minus one. \textit{A ribbon} is a skew diagram $\lambda/\mu$ that does not contain a $2\times2$ square.
  A ribbon can be partitioned into \textit{connected components}, where cells $x$ and $y$ lie in the same component if there is a sequence of cells $x=z_1, z_2,\dots,z_k=y$ in the ribbon in which $z_i$ and $z_{i+1}$ share an edge. A connected component of a ribbon with $t$ cells is called a \textit{$t$-hook.} 
  
  \begin{theorem}(Murnaghan-Nakayama Rule).\cite[Theorem 4.10.2]{Sa} 
  	Given positive integers $m$ and $n$, let $\delta \in \mathfrak{S}_{m+n}$ be an $m$-cycle and let $\pi$ be a permutation of the remaining $n$ elements. Then for every $\lambda\vdash m+n$,
  	$$[\lambda](\pi\delta)=\sum(-1)^{\mathrm{ht}(\lambda/\mu)}[\mu](\pi),$$ where the sum is over all $\mu \vdash n$ such that $\mu\subset \lambda$ and $\lambda/\mu$ is an $m$-hook.
  	
  \end{theorem}

Let $\lambda\vdash n$. Given a positive integer $t$, we say that $\lambda$ is \textit{$t$-singular} (resp. \textit{$t$-regular}), if at least one  part of $\lambda$ (resp. no part of $\lambda$) is divisible by $t$. 
We denote by $H^\lambda_{i,j}$ the \textit{$(i,j)$-hook} of $\lambda$, which consists of the $(i,j)$-cell, called the \textit{corner} of the $(i,j)$-hook, along with any other cells directly below or to the right of the corner. 
The \textit{hook length} of a cell $u=(i,j)\in \lambda$ is defined as the number of cells in $H^\lambda_{i,j}$.
The \textit{$t$-core of $\lambda$}, denoted by $\mathrm{Cor}_t(\lambda)$, is a diagram obtained by successive removals of $t$-hooks from $\lambda$.
The \textit{$t$-weight} $w_t(\lambda)$ of $\lambda$ is defined as $w_t(\lambda):=\frac{|\lambda|-|\mathrm{Cor}_t(\lambda)|}{t}$.
 The partition $\lambda$ is called \textit{$t$-core}, if it does not contain any cell of hook length $t$. We denote by $T(\lambda)$, the set of all positive integers $t$ such that $\lambda$ is a $t$-core partition. 

\begin{lemma}\label{core}
Let $\lambda\vdash n$. Then for every positive integer $t$, the following conditions are equivalent: 
\begin{itemize}
	\item[a)] $\lambda$ is a $t$-core partition. 
	\item[b)] $[\lambda](\alpha)=0$, for every $t$-singular partition $\alpha \vdash n$.
\end{itemize}
\end{lemma}
\begin{proof}
Let $\lambda$ be a $t$-core partition and let $\alpha=(\alpha_1, \alpha_2,\dots,\alpha_h)\vdash n$ be $t$-singular. Then for some integer $1\leq i\leq h$, the part $\alpha_i$ is divisible by $t$. Using \cite[2.7.40]{GA}, $\lambda$ is an $\alpha_i$-core partition. Hence, $\lambda$ does not have any cell of hook length $\alpha_i$. Therefore, it follows from the  Murnaghan-Nakayama rule that $[\lambda](\alpha)=0$. Conversely, suppose that $[\lambda](\alpha)=0$, for every $t$-singular partition $\alpha \vdash n$, but $\lambda$ is not a $t$-core partition. Then $w_t(\lambda)\geq 1$.
  If $\alpha:=(t^{w_t(\lambda)}, 1^{n-t.w_t(\lambda)})\vdash n$, then it follows from \cite[Corollary 2.7.33]{GA} that $[\lambda](\alpha)=c[\mathrm{Cor}_t(\lambda)]((1^{n-t.w_t(\lambda)}))$, for some integer $c\neq 0$. In particular, $[\lambda](\alpha)\neq 0$ which is a contradiction as $\alpha$ is $t$-singular. 
\end{proof}

\begin{lemma}\label{treangle}
Let $k$ be a positive integer. Then $T(\rho_k)$ is precisely the set of all positive integers $t$ such that either $t$ is even or $t\geq 2k+1$.
\end{lemma}
\begin{proof}
	The result immediately follows from this fact that the hook length of a cell $u=(i,j)\in \rho_k$ is an odd integer $h\leq 2k-1$. 
\end{proof}

Let $t\geq 2$ be an integer. A \textit{$t$-regular} (resp. \textit{$t$-singular}) element of the symmetric group $\mathfrak{S}_n$ is an element indexed by a $t$-regular (resp. $t$-singular) partition $\lambda\vdash n$. We denote by $\mathfrak{S}_n^{(t^\prime)}$ (resp. $\mathfrak{S}_n^{(t)}$), the set of $t$-regular (resp. $t$-singular) elements of $\mathfrak{S}_n$.  We can consider the restriction to $\mathfrak{S}_n^{(t^\prime)}$ of the scalar product on ordinary characters of $\mathfrak{S}_n$.
We set
$$\langle [\alpha],[\beta]\rangle_{\mathfrak{S}_n^{(t^\prime)}}:=\frac{1}{n!}\sum_{\pi\in \mathfrak{S}_n^{(t^\prime)}}[\alpha](\pi)[\beta](\pi),$$
\noindent for every irreducible characters $[\alpha]$ and $[\beta]$ of $\mathfrak{S}_n$.

The concept of generalized blocks was introduce by K$\ddot{u}$ lshammer, Olsson and Robinson in \cite{Ku}.
The irreducible characters $[\alpha]$ and $[\beta]$ are said to be \textit{directly $t$-linked} if $\langle [\alpha],[\beta]\rangle_{\mathfrak{S}_n^{(t^\prime)}}\neq 0$. The direct $t$-linking defines an equivalence relation (called $t$-linking) on the set of all irreducible characters of $\mathfrak{S}_n$ with equivalence classes that are called the \textit{linked $t$-blocks} of $\mathfrak{S}_n$. 
A combinatorial $t$-block of the symmetric group $\mathfrak{S}_n$ is the set of all irreducible characters indexed by partitions with the same $t$-core. When $t$ is a prime, combinatorial $t$-blocks are equal to $t$-blocks coming from modular representation theory (Nakayama conjecture). 
The concepts of linked and combinatorial $t$-blocks of $\mathfrak{S}_n$ are identical \cite[Theorem 5.13]{Ku}. Therefore, In abbreviation, we will use the term "$t$-block" instead of the terms "linked $t$-block" and "combinatorial $t$-block".


\section{Proof of our main results}

In this section, we wish to prove our main results. We start with a useful method to find constituents of $[\rho_k]^2$, which is inspired by semi-group property.
\begin{lemma}\label{main}
	Let $k$ and $m$ be positive integers with $k\geq m$.
	\begin{itemize}
		\item [a)] 
		For every $\alpha \in \{\lambda^{(a,b)}|\,\lambda \in \mathrm{Kron}(\rho_m),\,(a,b)\in G(k,m)\}$, the Kronecker coefficient $g(\rho_k,\rho_k, \alpha)$ is non-zero.
		\item [b)] If $k=2m$ or $2m+1$, 
		then for every $\alpha \in \{\lambda+\mu+\nu|\,\lambda,\mu \in \mathrm{Kron}(\rho_m),\,\nu \in \mathrm{Kron}(\sharp_{k-m})\}$,
		the Kronecker coefficient $g(\rho_k,\rho_k, \alpha)$ is non-zero.	
	\end{itemize}
\end{lemma}
\begin{proof}
a) 
Let $n\geq 2$ be a positive integer. First note that  the Kronecker coefficient $g((1^n),(1^n),(n))=1\neq 0$. It follows from semi-group property that for every $\alpha\in \mathrm{Kron}(\rho_{n-1})$, 
	$$g(\rho_{n}, \rho_n,\alpha^{(n,0)})=g(\rho_{n-1}+(1^n),\rho_{n-1}+(1^n),\alpha+(n)),$$
	\noindent and 
$$
g(\rho_{n}, \rho_n,\alpha^{(0,n)})=g(\rho_{n}, \rho_n,(\alpha^{(0,n)})^\prime)=g(\rho_{n-1}+(1^n), \rho_{n-1}+(1^n),\alpha^\prime+(n)),
$$
are non-zero. Thus, for every 
$\alpha\in \mathrm{Kron}(\rho_{n-1})$, we have that 
\begin{equation}
\alpha^{(n,0)},\alpha^{(0,n)}\in \mathrm{Kron}(\rho_{n}).
\end{equation}
Let $(a,b)$ be a $(k,m)$-good pair. There exists a subset $I\subseteq \{m+1,m+2,\dots,k\}$ such that $a=\sum_{i\in I}i$. We set $J:=\{m+1,m+2,\dots,k\}\backslash I$. For every $i\in I\cup J$, we define 
$$i^*:=\begin{cases} (i,0)  & \text{if }i\in I\\  (0,i)    & \text{if } i\in J 
\end{cases}.$$
\noindent 	Now let $\lambda\in \mathrm{Kron}(\rho_m)$. For the case $k=m$, the result follows from the definition of $\lambda^{(a,b)}$. thus, we can assume that $k>m$. It is easily checked that 
$$\lambda^{(a,b)}=(\dots((\lambda^{(m+1)^*})^{(m+2)^*})^{(m+3)^*}\dots)^{k^*}.$$
\noindent Therefore, the result follows from (1).\\

b) Let $\lambda,\mu\in \mathrm{Kron}(\rho_m)$ and $\nu \in \mathrm{Kron}(\sharp_{k-m}).$ Then using semi-group property, we deduce that the Kronecker coefficient 
$$g((\rho_m+\sharp_{k-m})^\prime,(\rho_m+\sharp_{k-m})^\prime,\mu+\nu)=g(\rho_m+\sharp_{k-m},\rho_m+\sharp_{k-m},\mu+\nu)$$
\noindent is non-zero. Hence, it again follows by semi-group property that the Kronecker coefficient
$$g(\rho_k, \rho_k,\lambda+\mu+\nu)=g(\rho_m+(\rho_m+\sharp_{k-m})^\prime,\rho_m+(\rho_m+\sharp_{k-m})^\prime,\lambda+(\mu+\nu))$$
\noindent is non-zero. Therefore, $\alpha:=\lambda+\mu+\nu\in \mathrm{Kron}(\rho_k)$.
\end{proof}

\noindent \textit{Proof of Proposition \ref{block}.}
Suppose that $B$ is the $t$-block of the symmetric group $\mathfrak{S}_{n}$ 
containing the irreducible character $[\alpha]$. 
If $|B|=1$, then as $B$ is a combinatorial $t$-block, we deduce that $\alpha$ is a $t$-core partition.
Thus, we can assume that $|B|\geq 2$.
Using Lemma \ref{core},
we get $$\sum_{\pi\in\mathfrak{S}^{(t)}_{n}}[\xi](\pi)[\beta](\pi)[\alpha](\pi)=0.$$ Thus, 
$$g(\xi,\beta,\alpha)=\langle[\xi][\beta], [\alpha]\rangle=\langle[\xi][\beta], [\alpha]\rangle_{\mathfrak{S}^{(t^\prime)}_{n}}.$$
\noindent Therefore, as $B$ is a linked $t$-block containing $[\alpha]$, we have that  
$$g(\xi,\beta,\alpha)=\sum_{\lambda\vdash n}g(\xi,\beta,\lambda)\langle[\lambda],[\alpha]\rangle_{\mathfrak{S}^{(t^\prime)}_{n}}=\sum_{[\lambda]\in B}g(\xi,\beta,\lambda)\langle [\lambda],[\alpha]\rangle_{\mathfrak{S}^{(t^\prime)}_{n}}.$$
\noindent Since $|B|\geq 2$ the partition $\alpha$ is not a $t$-core partition and so, it follows from Lemma \ref{core} that 
$$1-\langle [\alpha],[\alpha]\rangle_{\mathfrak{S}^{(t^\prime)}_{n}}=\langle [\alpha],[\alpha]\rangle-\langle [\alpha],[\alpha]\rangle_{\mathfrak{S}^{(t^\prime)}_{n}}$$
\noindent is non-zero. Hence, 
$$\sum_{[\lambda]\in B\backslash \{[\alpha]\}}g(\xi,\beta,\lambda)\langle [\lambda],[\alpha]\rangle_{\mathfrak{S}^{(t^\prime)}_{n}}=g(\xi,\beta, \alpha)(1-\langle [\alpha],[\alpha]\rangle_{\mathfrak{S}^{(t^\prime)}_{n}})$$
\noindent is non-zero. Therefore, the result follows from this fact that $B$ is a combinatorial $t$-block. \qed\\

\noindent \textit{Proof of Corollary \ref{hook}.}
If $\alpha$ is a hook partition, then it follows from \cite[Theorem 2.1]{Re} that $g(\alpha,\alpha,\alpha)$ or $g(\alpha,\alpha,\alpha^\prime)$ is non-zero. Therefore, without loss of generality, we can assume that $\alpha$ is not a hook partition. At first, let $\alpha\neq \alpha^\prime$. It is easily checked that $g(\alpha,\alpha,(n))=1$ and $g(\alpha,\alpha,(1^n))=0$. Note that the $n$-block of the symmetric group $\mathfrak{S}_n$ is precisely the set of all irreducible characters indexed by hook partitions of $n$. Thus, applying Proposition \ref{block}, we deduce that there exists a hook partition $\beta\neq (n),(1^n)$ such that $g(\alpha,\alpha,\beta)\neq 0$. Now, let $\alpha=\alpha^\prime=(\alpha_1,\alpha_2,\dots,\alpha_h)$. Set $\alpha^*=(\alpha_2,\alpha_3,\dots,\alpha_h)^\prime$. Since $\alpha^*$ is not self-conjugate, there exists a hook partition $\beta^*\neq (n-\alpha_1),(1^{n-\alpha_1})$ such that $g(\alpha^*,\alpha^*,\beta^*)\neq 0$. Therefore as $g((1^{\alpha_1}),(1^{\alpha_1}),(\alpha_1))=1$, it follows from semi-group property that $g(\alpha,\alpha,\beta^*+(\alpha_1))\neq 0$. This completes the proof. \qed\\

\noindent \textit{Proof of Theorem \ref{saxl}.}
a) Let $m:=\lfloor\log_2 k\rfloor$. 
Given $\alpha\in \mathrm{Par}_T(k)$, there are positive integers $s$ and $t$ with $0\leq s\leq t\leq m-1$, the pair $(a,b)\in G(k,2^{t+1})$, the partition $\beta_s\in \mathrm{kron}(\rho_{2^s})$ and for every integer $s\leq i\leq t$, $\alpha_i\in \mathrm{kron}(\rho_{2^i})$ such that  $$\alpha=(\beta_s+\sum_{i=s}^{t}(\sharp_{2^i}+\alpha_i))^{(a,b)}.$$
 Note that it follows from \cite[Corollary 3.2]{Bs} that  $\sharp_{2^i}\in \mathrm{Kron}(\sharp_{2^i})$, for every positive integer $s\leq i\leq t$. For the case $m=0$, the result is clear. Thus, we can assume that $k\geq 2$. We first claim that $\lambda(t):=\beta_s+\sum_{i=s}^{t}(\sharp_{2^i}+\alpha_i)\in \mathrm{Kron}(\rho_{2^{t+1}})$. We prove it by induction on $t$. If $t=s$, then $\lambda(s)=\beta_s+\alpha_s+\sharp_{2^s}$, and so the result follows from Lemma \ref{main}(b). Now suppose that $t>s$. Given an integer $s\leq r<m-1$, we assume that the claim holds for the case $t=r$ and we try to prove it for $t=r+1$. By the induction hypothesis, $\lambda(r)\in \mathrm{Kron}(\rho_{2^{r+1}})$. Hence it follows from Lemma \ref{main}(b) that $\lambda(r+1)=\lambda(r)+\alpha_{r+1}+\sharp_{2^{r+1}}\in \mathrm{Kron}(\rho_{2^{r+2}})$. Therefore, as $\alpha=\lambda(t)^{(a,b)}$, the result follows from Lemma \ref{main}(a).

b) 
By Lemma \ref{treangle}, $t\in T(\rho_k)$. Therefore, the result follows from Proposition \ref{block}. \qed\\
 \begin{example}
	Let $k$ be a positive integer, and let $m:=\lfloor\log_2 k\rfloor$. Given a positive integer $1\leq i\leq m-1$, suppos that $\alpha_i:=((2^{i-1})^2,1^{2^{2i-1}-2^{i-1}})\vdash\binom{2^i+1}{2}$. Then applying \cite[Theorem 1.4]{Morotti}, $\alpha_i$ is of odd degree. Hence, it follows from \cite[Theorem 5.2]{BO} that $\alpha_i\in \mathrm{Kron}(\rho_{2^i})$. Therefore, using Theorem \ref{saxl} (a), we have that for every $(a,b)\in G(k,2^m)$, the Kronecker coefficient $g(\rho_k,\rho_k,(\rho_1+\sum_{i=0}^{m-1}(\sharp_{2^i}+\alpha_{i}))^{(a,b)})$ is non-zero.
\end{example}
We end this section with the proof of Corollary \ref{telescope}. 

\noindent\textit{ Proof of Corollary \ref{telescope}}.
 Applying  \cite[Corollary 3.2]{Bs}, we have that $\rho_{2^i} \in \mathrm{Kron}(\rho_{2^i})$, for every non-negative integer $i$. Therefore, the result follows from Theorem \ref{saxl}(a).\qed


\section*{Declaration of interests}
I declare that I have no known competing financial interests or personal relationships that could have appeared to influence the work reported in this note.

\section*{Data Availability Statements}
Data sharing is not applicable to this article, as no data sets were generated or analyzed during the current study.


\section*{Acknowledgements}
I would like to express my gratitude to Lucia Morotti for her valuable comments on $t$-core partitions. This research was supported in part
by a grant from School of Mathematics, Institute for Research in Fundamental Sciences (IPM).


\end{document}